\documentclass[12pt]{article}
\usepackage{a4wide}
\usepackage{amsmath,amssymb,amsthm}
\usepackage[T2A]{fontenc}
\usepackage{mathtools}
\usepackage{tikz}
\usetikzlibrary{positioning,automata}
\usetikzlibrary{arrows,calc}

\author{Ievgen Bondarenko}
\title{\textbf{Self-similar groups and the zig-zag and replacement products of graphs}}

\newcommand{\zz}{\mathop{\mbox{\textcircled{$z$}}}}
\newcommand{\rp}{\mathop{\mbox{\textcircled{$r$}}}}
\newcommand{\diam}{{\rm diam}}
\newcommand{\graph}{\mathcal{G}}
\newcommand{\GGS}{\textsf{GGS}}

\newtheorem{theorem}{Theorem}

\newtheorem{corollary}{Corollary}[theorem]

\theoremstyle{definition}

\newtheorem{example}{Example}

\begin{document}
\maketitle

\begin{flushright}
\end{flushright}

\begin{abstract}

Every finitely generated self-similar group naturally produces an infinite sequence of
finite $d$-regular graphs $\Gamma_n$. We construct self-similar groups, whose graphs
$\Gamma_n$ can be represented as an iterated zig-zag product and graph powering:
$\Gamma_{n+1}=\Gamma_n^k\zz\Gamma$ ($k\geq 1$). Also we construct self-similar groups, whose graphs
$\Gamma_n$ can be represented as an iterated replacement product and graph powering:
$\Gamma_{n+1}=\Gamma_n^k\rp\Gamma$ ($k\geq 1$). This gives simple explicit examples of
self-similar groups, whose graphs $\Gamma_n$ form an expanding family, and examples of automaton groups, whose graphs $\Gamma_n$ have linear diameters $\diam(\Gamma_n)=O(n)$ and bounded girth.

\vspace{0.1cm}\textit{2010 Mathematics Subject Classification}: 05C25, 20F65, 20E08

\end{abstract}


\section{Introduction}

A sequence of finite $d$-regular
graphs $(\Gamma_n)_{n\geq 1}$ is an expanding family if there exists $\varepsilon>0$
such that $\lambda(\Gamma_n)<1-\varepsilon$ for all $n\in\mathbb{N}$, where $\lambda(\Gamma)$ is the second largest (in absolute value) eigenvalue of the
normalized adjacency matrix of $\Gamma$. Expanding graphs have many interesting applications in different areas of mathematics and computer science (see \cite{wigderson} and the references therein). That is why many constructions of expanding families were proposed for the last decades, most of which have algebraic nature.

In \cite{RVW00}, Reingold, Vadhan, and Wigderson discovered a simple combinatorial construction of expanding graphs. Their construction is based on the new operation on regular graphs --- the zig-zag product $\zz$.
The estimates on the second eigenvalue of the zig-zag product of graphs proved in \cite{RVW00} lead to the constuction of expanders as an iterated zig-zag product and graph squaring: the sequence $\Gamma_{n+1}=\Gamma_n^2\zz\Gamma$, $\Gamma_1=\Gamma^2$ is an expanding family if $\lambda(\Gamma)$ is small enough.
Later, the zig-zag product showed its effectiveness in constructing graphs with other exceptional properties,
various codes, in computational complexity theory, etc.

The zig-zag product is directly related to the simpler replacement product $\rp$, which replace every vertex of one graph by a copy of another graph. This product was widely used in various contexts. For example, the replacement product of the graph of the $d$-dimensional cube and the cycle on $d$ vertices is the so-called cube-connected cycle, which is used in the network architecture for parallel computations. Gromov \cite{gromov} considered the graphs of $d$-dimensional cubes for different dimensions and estimated the second eigenvalue of their iterated replacement product (iterated cubical graphs). Previte \cite{previte} studied the convergence of iterated replacement product $\Gamma_{n+1}=\Gamma_n\rp\Gamma$, normalized to have diameter one, in the Gromov-Hausdorff metric and their limit spaces. The estimate on the second eigenvalue of the replacement product of graphs proved in \cite{RVW00} leads to the expanding family $\Gamma_{n+1}=\Gamma_n^4\rp\Gamma$ when $\lambda(\Gamma)$ and $\lambda(\Gamma_1)$ are small enough \cite{ZZ_RP_LDPC}.

In this paper we establish a connection between the zig-zag and replacement products of graphs and self-similar groups.
The theory of self-similar groups \cite{self_sim_groups} was developed from several examples of groups (mainly the Grigorchuk group) that enjoy many extreme properties (intermediate
growth, finite width, non-uniformly exponential growth, periodic groups, amenable but
not elementary amenable groups, just-infinite groups, etc.) Self-similar groups are
specific groups of transformations on the space of all finite words over an alphabet
that preserve the length of words. Every self-similar group can be easily defined by a finite system of wreath recursions, while properties of the group remain mysterious.

By fixing a generating set of a self-similar group,
we get a sequence of $d$-regular graphs $\Gamma_n$ associated to the action of
generators on words of length $n$. A natural question arises whether we can produce an expanding family in this way.
However, the graphs $\Gamma_n$ were studied mostly for the opposite case of contracting
self-similar groups. In this case, the graphs $\Gamma_n$ converge in certain sense to a
compact fractal space, which lead to the notion of a limit space of a contracting
self-similar group and further developed into the beautiful theory of iterated
monodromy groups \cite{self_sim_groups}. The diameter of graphs $\Gamma_n$
for contracting groups has exponential growth in terms of $n$ (polynomial in the number of vertices), what makes them opposite to expanding
graphs and the zig-zag product.

An important class of self-similar groups is the class of automaton groups. These
groups are given by finite-state transducers (Mealy automata) with the same input and
output alphabets. Every state of such an automaton $A$ produces a transformation of
words over an alphabet. If all these transformations are invertible, they generate a
self-similar group under composition of functions called the automaton group $G_A$
generated by $A$. For example, the Grigochuk group is generated by a $5$-state
automaton over a $2$-letter alphabet.

The graphs $\Gamma_n$ for an automaton group $G_A$ can be expressed through the standard operation of composition of automata, namely $\Gamma_n=\widehat{A}\circ\ldots\circ\widehat{A}$ ($n$ times), where $\widehat{A}$ is the dual automaton. However, expanding properties of automata composition are unknown.
The complete spectrum of graphs $\Gamma_n$ was computed only for a
few automaton groups \cite{Hecke,SpectralBasilica,gri_sunik:spectrum}, and the
general case remains widely open.
Nevertheless, Glasner and Mozes \cite{square_compl} realized certain groups with property (T) as automaton goups, what implies that the associated graphs $\Gamma_n$ form an expanding family.
The corresponding automata are large and were not described explicitly.
At the same time, there are two specific $3$-state automata over a $2$-letter alphabet, the Aleshin and Bellaterra automata, whose graphs $\Gamma_n$ form asymptotic expanders
\cite[Section~10]{Gri:dynamics}, and the question is raised \cite[Problem
10.1]{Gri:dynamics} whether actually these graphs are expanders. This problem remains
open. Even the asymptotic of diameters of $\Gamma_n$ for these two automata is unknown; the best known upper bound is $O(n^2)$ \cite{Pak}.

In this paper, given $k\geq 1$ and a graph $\Gamma$ with certain restrictions, we construct self-similar groups, whose graphs $\Gamma_n$ can be
represented as iterated zig-zag or replacement products and graph powering:
$\Gamma_{n+1}=\Gamma_n^k\zz\Gamma$ for all $n\geq 1$ or $\Gamma_{n+1}=\Gamma_n^k\rp\Gamma$ for all $n\geq 1$. This gives explicit examples of
self-similar groups whose graphs $\Gamma_n$ form a family of expanders. The established
connection between self-similar groups and the zig-zag product is not surprising --- the
zig-zag product of graphs is closely related to the semidirect product of groups
\cite{semidirect}, while self-similar groups to the wreath product of groups. We also
note that our construction modeling iterated zig-zag product is a self-similar analog of the construction from
\cite{iterative}. In the case $k=1$, the constructed groups are automaton groups. This gives simple explicit examples of automaton groups whose graphs $\Gamma_n$ have linear diameters $O(n)$ (logarithmic in the number of vertices) and bounded girth. Interestingly, some of the automaton groups modeling iterated replacement product belong to the class of $\GGS$ groups \cite{PhD_Barth,branch_groups}. In particular, these groups are not finitely presented and have intermediate growth.

\section{The zig-zag and replacement products of graphs}

All graphs in this paper are regular, undirected, and may have loops and multiple edges.

Let $\graph$ be a $D$-regular graph
on $N$ vertices and let $\Gamma$ be a $d$-regular graph on $D$ vertices. We label the
edges near every vertex of $\graph$ by the vertices of $\Gamma$ in one-to-one fashion;
for $v\in V(\graph)$ and $x\in V(\Gamma)$, let $v[x]$ be the $x$-neighbor of $v$. If an
edge is labeled by $x$ near $v$ and by $y$ near $u$, i.e., $v[x]=u$ and $u[y]=v$, we write $v
\begin{tikzpicture} \draw (0.1,0.1) -- (0.9,0.1); \node [above] at (0.5,0)
{\scriptsize{$x$ \ \ $y$}};\end{tikzpicture} u$. The zig-zag and replacement products depend on the chosen labeling.

\vspace{0.2cm}\textbf{The zig-zag product.} The \textit{zig-zag product} $\graph\zz \Gamma$
is a $d^2$-regular graph on $ND$ vertices $V(\Gamma)\times V(\graph)$. The edges of
$\graph\zz \Gamma$ are formed by ``zig-zag'' paths of length three:
\begin{enumerate}
    \item for every edge $x\mbox{ --- }x'$ in $\Gamma$ (the zig-step),\vspace{-0.2cm}
    \item the edge $v
\begin{tikzpicture} \draw (0.1,0.1) -- (0.9,0.1); \node [above] at (0.55,0)
{\scriptsize{$x'$ \ \ $y'$}};\end{tikzpicture} u$ in $\graph$,
    \item and every edge $y'\mbox{ --- }y$ in $\Gamma$ (the zag-step),
\end{enumerate}
there is an edge between $(x,v)$ and $(y,u)$ in $\graph\zz \Gamma$. (Classically, the
vertices of the zig-zag product are written as pairs $(v,x)$. We switched the order to
show a similarity with action graphs of self-similar groups. As usual, by switching
from right to left, we get a connection between two object studied in different
contexts.)


The next basic properties easily follow from the definition.
The zig-zag product of any two graphs has girth $\leq 4$ and diameter $\diam(\graph\zz
\Gamma)\leq \diam(\graph)+2\diam(\Gamma)$. The zig-zag product of connected graphs is not always connected; one easy sufficient condition is the following: If any two vertices of $\Gamma$ can be connected by a path of even length, then the graph $\graph\zz\Gamma$ is connected for any connected graph $\graph$.

Many applications of the zig-zag product are based on the following spectral property proved in \cite{RVW00}:
\begin{equation}\label{eqn_zig_zag_lambda}
\lambda(\graph\zz \Gamma)\leq \lambda(\graph)+\lambda(\Gamma)+\lambda(\Gamma)^2,
\end{equation}
where $\lambda(\Gamma)$ is the second largest (in absolute value) eigenvalue of the
normalized adjacency matrix of $\Gamma$.

\vspace{0.2cm}\textbf{The replacement product.} The \textit{replacement product} $\graph\rp \Gamma$
is a $(d+1)$-regular graph on $ND$ vertices $V(\Gamma)\times V(\graph)$ with the following edges:
\begin{enumerate}
    \item for every edge $x\mbox{ --- }y$ in $\Gamma$ and $v\in V(\graph)$ there is an edge $(x,v)\mbox{ --- }(y,v)$ in $\graph\rp \Gamma$;
    \item for every edge $v
\begin{tikzpicture} \draw (0.1,0.1) -- (0.9,0.1); \node [above] at (0.55,0)
{\scriptsize{$x$ \ \ $y$}};\end{tikzpicture} u$ in $\graph$ there is an edge $(x,v)\mbox{ --- }(y,u)$ in $\graph\rp \Gamma$.
\end{enumerate}
In other words, we replace each vertex $v$ of $\graph$ with a copy of $\Gamma$ (keeping all the edges of $\Gamma$ in all the copies), and adjoin edges adjacent to $v$ in $\graph$ to the corresponding vertices of $\Gamma$ using the chosen one-to-one correspondence between these edges and vertices of $\Gamma$.

The next properties easily follow from the definition. The replacement product of connected graphs is connected, the diameter satisfies $\diam(\graph\rp\Gamma)\leq \diam(\graph)\cdot\diam(\Gamma)$, and the girth of $\graph\rp\Gamma$ is not greater than the girth of $\Gamma$.

In \cite{RVW00}, the expansion property of the replacement product is estimated as
\begin{equation}\label{eqn_rp_lambda}
\lambda(\graph\rp \Gamma)\leq \left( p+(1-p)(\lambda(\graph)+\lambda(\Gamma)+\lambda(\Gamma)^2)  \right)^{1/3},
\end{equation}
where $p=d^2/(d+1)^3$.

\vspace{0.2cm}\textbf{Iterative construction of expanders.}
Let us describe the construction of
expanding families using the zig-zag product and graph powering presented in
\cite{RVW00}. Take a $d$-regular
graph $\Gamma$ on $d^4$ vertices such that $\lambda(\Gamma)\leq 1/5$ (such graphs exist
by probabilistic arguments). Define the sequence of graphs $(\Gamma_n)_{n\geq 1}$ as
follows:
\begin{equation}\label{eqn_expanders_zig_zag}
\Gamma_1=\Gamma^2, \qquad \Gamma_{n+1}=\Gamma_n^2\zz \Gamma, n\geq 1.
\end{equation}
(The \textit{$k$-th power} $\Gamma^k$ of a graph $\Gamma$ is the graph on the vertices of $\Gamma$,
where each edge corresponds to a path of length $k$ in $\Gamma$. Note that $\lambda(\Gamma^k)=\lambda(\Gamma)^k$.)
Then the estimate (\ref{eqn_zig_zag_lambda}) implies that the graphs $\Gamma_n$ are
$d^2$-regular graphs with $\lambda(\Gamma_n)\leq 2/5$.

Analogous construction works with the replacement product as well \cite{ZZ_RP_LDPC}.
Take a $(d+1)$-regular graph $\Gamma_1$ and a $d$-regular
graph $\Gamma$ on $(d+1)^4$ vertices such that $\lambda(\Gamma_1)\leq 1/5$, $\lambda(\Gamma)\leq 1/5$. Define the sequence of graphs $(\Gamma_n)_{n\geq 1}$ as
follows:
\begin{equation}\label{eqn_expanders_rp}
\Gamma_{n+1}=\Gamma_n^4\rp \Gamma, n\geq 1.
\end{equation}
The estimate (\ref{eqn_rp_lambda}) implies that the graphs $\Gamma_n$ are
$(d+1)$-regular graphs with $\lambda(\Gamma_n)\leq 1/10$.

\section{Self-similar groups and their action graphs}

Every finitely generated self-similar group
can be given by a finite system (wreath recursion)
\begin{equation}\label{eqn_wreath_recursion}
\left\{%
\begin{array}{ll}
s_1&=\pi_1(w_{11},w_{12},\ldots,w_{1d}) \\
s_2&=\pi_2(w_{21},w_{22},\ldots,w_{2d}) \\
\vdots & \qquad \vdots \\
s_k&=\pi_k(w_{k1},w_{k2},\ldots,w_{kd})
\end{array}%
\right.,
\end{equation}
where $\pi_i$ is a permutation on $X=\{1,2,\ldots,d\}$ and $w_{ij}$ is a word over
$S\cup S^{-1}$, $S=\{s_1,s_2,\ldots,s_k\}$. The system defines the action of $S$ on the
set $X^{*}$ of all finite words over $X$ (we use left actions). Each $s_i$ acts on $X$
by the permutation $\pi_i$, and the action on words over $X$ is defined by the
recursive rule
\[
s_i(xv)=\pi_i(x)w_{ix}(v), \quad x\in X, v\in X^{*},
\]
where $w_{ix}$ acts by composition. These transformations are invertible, and the group
generated by them under composition is called the \textit{self-similar group $G=\langle
S\rangle$ associated to the system} (\ref{eqn_wreath_recursion}).

When all words $w_{ix}$ in (\ref{eqn_wreath_recursion}) are letters, i.e.,
$w_{ix}=s_{ix}\in S$, the system (\ref{eqn_wreath_recursion}) can be represented by
a finite-state automaton-transducer $A$ over the alphabet $X$ with states $S$. The automaton $A$ is represented by a finite directed graph with vertices $S$ and arrows
$s_i\rightarrow s_{ix}$ labeled by $x|\pi_i(x)$ for all $x\in X$ and $s_i\in S$.
The action of $s_i$ on $X^{*}$ can be described using the automaton $A$ as follows. Given a word
$v=x_1x_2\ldots x_n\in X^{*}$, there exists a unique directed path in the automaton
$A$ starting at the state $s_i$ and labeled by $x_1|y_1$, $x_2|y_2$,\ldots, $x_n|y_n$ for
some $y_i\in X$. Then the word $y_1y_2\ldots y_n$ is the image of
$x_1x_2\ldots x_n$ under $s_i$.
In this case, the group $G=\langle S\rangle$ is called the \textit{automaton group} given by the automaton~$A$.

Self-similar groups
preserve the length of words under the action on $X^{*}$, and we can restrict the
action to $X^n$, words of length $n$, for each $n\in\mathbb{N}$. By choosing a finite
symmetric generating set $S$ of a self-similar group $G$, we get a sequence of
$|S|$-regular graphs $(\Gamma_n)_{n\geq 1}$ of the action on $X^n$, $n\geq 1$. The vertex set of
$\Gamma_n$ is $X^n$, and for every $s\in S$ and $v\in X^n$ there is an edge between the
vertices $v$ and $s(v)$. The graph $\Gamma_n$ is a Schreier coset graph of $G$ if the
group acts transitively on $X^n$.

If the generating set $S$ is given by the system (\ref{eqn_wreath_recursion}), the
graphs $\Gamma_n=\Gamma_n(S)$ can be constructed iteratively, very similar to
(\ref{eqn_expanders_zig_zag}). Every $s_i\in S$ produces an edge $xv\mbox{ ---
}\pi_i(x)w_{ix}(v)$ in $\Gamma_n$, which can be interpreted as a zig-step $x\mbox{
--- }\pi_i(x)$ in the graph $\Gamma_1$, and a walk $v\mbox{ --- }w_{ix}(v)$ in the graph
$\Gamma_{n-1}$. In contrast to the zig-zag product, we are missing the
zag-step (and the walk $w_{ix}$ is not agreed with $\pi_i(x)$), but we will see in the next section that one can make the edge $x\mbox{ --- }\pi_i(x)$ to be already the combination of the zig and zag steps.

\section{Modeling iterated zig-zag and replacement products by automaton groups}

In this section we construct automaton groups whose action graphs satisfy $\Gamma_{n+1}=\Gamma_n\zz \Gamma$  for all $n\geq 1$ or $\Gamma_{n+1}=\Gamma_n\rp \Gamma$ for all $n\geq 1$, where $\Gamma$ is a fixed graph.

\vspace{0.2cm}\textbf{Modeling iterated zig-zag product.} Let $X=\{1,2,\ldots,d\}$ and $P$ be a symmetric set of permutations on $X$ such that $d=|P|^2$. We introduce formal symbols $s_{(\pi,\tau)}$ for $\pi,\tau\in P$ and define wreath recursion~(\ref{eqn_wreath_recursion}) for the set $S_P=\{
s_{(\pi,\tau)} : \pi,\tau\in P\}$, $|S_P|=d$ as follows. Choose an order on $S_P$: let
$S_P=\{s_1,\ldots,s_{d}\}$. Let $\gamma$ be the permutation on $X$ given by the rule: if
$s_x=s_{(\pi,\tau)}$ then $s_{\gamma(x)}=s_{(\tau^{-1},\pi^{-1})}$. Notice that
$\gamma=\gamma^{-1}$. Define wreath recursion by
\[
s_{(\pi,\tau)}=\tau\gamma
(s_1,s_2,\ldots,s_{d})\pi=\tau\gamma\pi(s_{\pi(1)},s_{\pi(2)},\ldots,s_{\pi(d)}), \
\pi,\tau\in P,
\]
(here $\pi$ and $\tau$ will play a role of the zig and zag steps respectively). Let $G_{P}$ be the
self-similar group defined by this recursion. It is important to note that $S_P$
defines a symmetric generating set of $G_{P}$, where
$s^{-1}_{(\pi,\tau)}=s_{(\tau^{-1},\pi^{-1})}$ (in other notations,
$s^{-1}_x=s_{\gamma(x)}$). This follows inductively from the recursions
\begin{align*}
s^{-1}_{(\pi,\tau)}(xv)&=\pi^{-1}\gamma^{-1}\tau^{-1}(x)\, s^{-1}_{\gamma^{-1}\tau^{-1}(x)}(v),\\
s_{(\tau^{-1},\pi^{-1})}(xv)&=\pi^{-1}\gamma\tau^{-1}(x)\, s_{\tau^{-1}(x)}(v),
\end{align*}
$x\in X$, $v\in X^{*}$ (use $\gamma=\gamma^{-1}$).

\begin{theorem}
The action graphs $\Gamma_n$ of the group $G_P=\langle S_P\rangle$ satisfy
$\Gamma_{n+1}=\Gamma_n\zz \Gamma$, $n\geq 1$, where $\Gamma$ is the graph of the action
of $P$ on $X$.
\end{theorem}
\begin{proof}
The graph $\Gamma$ is a $|P|$-regular graph on $d$ vertices, while $\Gamma_n$ are $d$-regular graphs. In order to define the zig-zag product $\Gamma_n\zz\Gamma$, we should label the edges of $\Gamma_n$ by the vertices of $\Gamma$.
For $x\in X$, define the
$x$-neighbor of a vertex $v\in X^n$ as $v[x]:=s_x(v)$. In this way we get the labeling
of edges $v
\begin{tikzpicture} \draw (0.1,0.1) -- (1.1,0.1); \node [above] at (0.6,0)
{\scriptsize{$x$ \ $\gamma(x)$}};\end{tikzpicture} s_x(v)$. Now we can consider the
zig-zag product $\Gamma_n\zz\Gamma$. The vertex set of $\Gamma_n\zz\Gamma$ can be
naturally identified with the vertex set $X^{n+1}$ of $\Gamma_{n+1}$ via $(x,v)\leftrightarrow xv$. For every zig-zag path
\begin{center}
$x\overset{\pi}{\mbox{ --- }} x'=\pi(x)$ \mbox{ in $\Gamma$}, \quad $v
\begin{tikzpicture} \draw (0.1,0.1) -- (0.9,0.1); \node [above] at (0.5,0)
{\scriptsize{$x'$ \ $y'$}};\end{tikzpicture} v[x'] \mbox{ in $\Gamma_n$}$, \quad
$y'\overset{\tau}{\mbox{ --- }} y=\tau(y') \mbox{ in $\Gamma$}$
\end{center}
there is an edge $xv-yv[x']$ in $\Gamma_n\zz\Gamma$. Here
$y'=\gamma(x')=\gamma(\pi(x))$ and $v[x']=s_{x'}(v)=s_{\pi(x)}(v)$. Therefore this edge
is precisely the edge of $\Gamma_{n+1}$ given by $s_{(\pi,\tau)}$:
\begin{align*}
xv \mbox{ --- } \tau(\gamma(\pi(x))) s_{\pi(x)}(v).
\end{align*}
\end{proof}


The next statement immediately follows from the properties of the zig-zag product.

\begin{corollary}
The action graphs $\Gamma_n$ of the group $G_P=\langle S_P\rangle$ have bounded girth and linear diameters $\diam(\Gamma_n)=O(n)$.
If $\Gamma_1$ is connected ($P\gamma P$ acts
transitively on $X$) and there is a path of even length between any two vertices of $\Gamma$, then all graphs $\Gamma_n$ are connected (the group $G_P$ acts transitively on $X^n$).
\end{corollary}

The wreath recursion for the group $G_P$ defines a finite automaton over $X$ with $d$ states. Therefore every $G_P$ is an automaton group.

\begin{example}
Let $d=4$, $X=\{1,2,3,4\}$ and $P=\{(1\,2), (1\,4)(2\,3)\}$. Then $\gamma=(2\,3)$ and the group $G_P$ is generated by $s_1,s_2,s_3,s_4$ given by the wreath recursion:
\begin{align*}
s_1&=(1\,2)(2\,3)(s_1,s_2,s_3,s_4)(1\,2)=(1\,3)(s_2,s_1,s_3,s_4)\\
s_2&=(1\,4)(2\,3)(2\,3)(s_1,s_2,s_3,s_4)(1\,2)=(1\,2\,4)(s_2,s_1,s_3,s_4)\\
s_3&=(1\,2)(2\,3)(s_1,s_2,s_3,s_4)(1\,4)(2\,3)=(1\,4\,2)(s_4,s_3,s_2,s_1)\\
s_4&=(1\,4)(2\,3)(2\,3)(s_1,s_2,s_3,s_4)(1\,4)(2\,3)=(2\,3)(s_4,s_3,s_2,s_1)
\end{align*}
The generating automaton is shown on the left-hand side of Figure~\ref{fig_automata}.
\end{example}

\begin{figure}
\begin{minipage}[b]{0.49\textwidth }
\begin{center}
{\small
\begin{tikzpicture}[shorten >=1pt,node distance=3.5cm,on grid,auto,/tikz/initial text=,semithick]
   \begin{scope}  
   (0,0) \node[state] at (2:400) (s1)   {$s_1$};
   \end{scope}
   -- (-1,0) \node[state] (s4) [below=of s1] {$s_4$};
   (1,2) \node[state] (s2) [right=of s1] {$s_2$};
   (1,2) \node[state] (s3) [below=of s2] {$s_3$};
    \path[->]
    (s1) edge [loop left] node  {$2|2$} (s1)
    (s1) edge [bend left] node {$1|3$} (s2)
    (s1) edge  node {$4|4$} (s4)
    (s2) edge  node  [above]{$2|4$} (s1)
    (s2) edge [loop right] node {$1|2$} (s2)
    (s2) edge [bend left] node  {$3|3$} (s3)
    (s3) edge  node {$3|3$} (s2)
    (s3) edge [loop right] node  {$2|1$} (s3)
    (s3) edge [bend left] node [below]{$1|4$} (s4)
    (s4) edge [bend left] node {$4|4$} (s1)
    (s4) edge  node  [below]{$2|3$} (s3)
    (s4) edge [loop left] node {$1|1$} (s4);
    \path[->,min distance=1cm] (s1)edge[in=120,out=330,above] node [pos=0.25]{$3|1$} (s3);
    \path[->,min distance=1cm] (s3)edge[in=300,out=150,above] node [pos=0.25,below]{$4|2$} (s1);
    \path[->,min distance=1cm] (s2)edge[in=60,out=210,above] node [pos=0.25,above]{$4|1$} (s4);
    \path[->,min distance=1cm] (s4)edge[in=240,out=30,above] node [pos=0.25,below]{$3|2$} (s2);
\end{tikzpicture}
}\vspace{0.4cm}
\end{center}
\end{minipage}
\begin{minipage}[b]{0.49\textwidth }
\begin{center}
\small{
\begin{tikzpicture}[shorten >=1pt,node distance=3.5cm,on grid,auto,/tikz/initial text=,semithick]
   \begin{scope}  
   (0,0) \node[state] at (2:400) (s1)   {$s$};
   \end{scope}
   -- (-1,0) \node[state] (s4) [below=of s1] {$\sigma$};
   (1,2) \node[state] (s2) [right=of s1] {$\sigma^{-1}$};
   (1,2) \node[state] (e) [below=of s2] {$e$};
    \path[->]
    (s1) edge [loop left] node  {$3|3$} (s1)
    (s1) edge  node {$2|1$} (s2)
    (s1) edge  node {$1|2$} (s4)
    (s2) edge [bend left] node  {$2|1$} (e)
    (s2) edge  node {$3|2$} (e)
    (s2) edge [bend right] node  {$1|3$} (e)
    (e) edge [loop right] node {$1|1$} (e)
    (e) edge [in=-60, out=-30, loop] node  {$2|2$} (e)
    (e) edge [loop below] node {$3|3$} (e)
    (s4) edge [bend left] node {$1|2$} (e)
    (s4) edge  node  {$2|3$} (e)
    (s4) edge [bend right] node {$3|1$} (e);
\end{tikzpicture}
}
\end{center}
\end{minipage}

\caption{The generating automata for two examples of groups $G_P$ and $G_Q$}\label{fig_automata}
\end{figure}
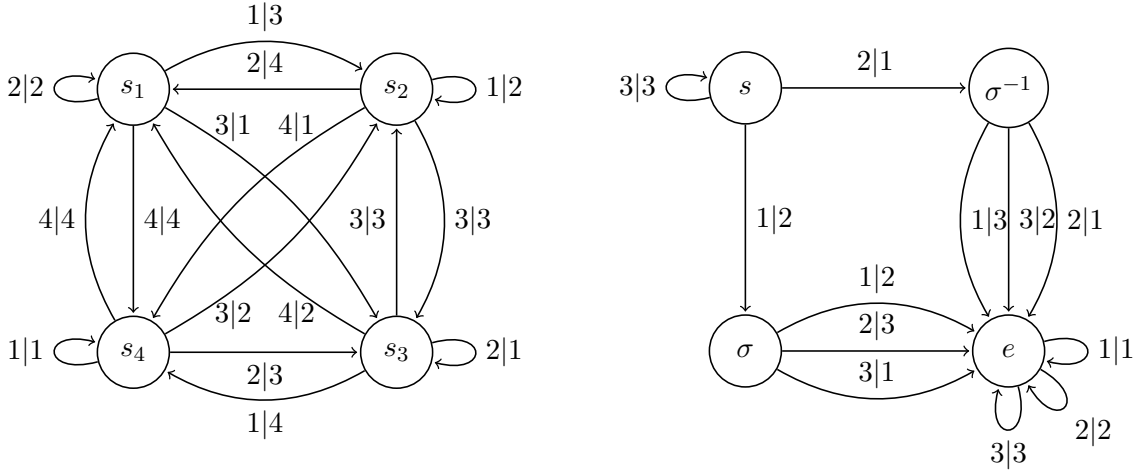

The construction can be modified for the case $d>|P|^2$. We add $d-|P|^2$ empty words
$e$ to the wreath recursion:
\[
s_{(\pi,\tau)}=\tau\gamma(s_1,s_2,\ldots,s_{|P|^2},e,\ldots,e)\pi,
\]
where $e$ acts trivially on $X^{*}$. Then the action graphs satisfy
$\Gamma_{n+1}=\Gamma_n^{\circ}\zz\Gamma$, where $\Gamma_n^{\circ}$ is obtained from
$\Gamma_n$ by adding $d-|P|^2$ loops to every vertex.

\vspace{0.2cm}\textbf{Modeling iterated replacement product.}
Let $Q=\{\pi_1,\ldots,\pi_d\}$ be a symmetric set of permutations on $X=\{1,2,\ldots,d+1\}$.
Let $\gamma$ be the involution on $X$ given by the rule: $\pi_{\gamma(x)}=\pi^{-1}_{x}$ and $\gamma(d+1)=d+1$.
We define wreath recursion for the set $S_Q=\{ s_1,\ldots, s_{d+1}\}$ by
\begin{eqnarray*}
s_i&=&\pi_i(e,e,\ldots,e), \quad \mbox{ $i=1,2,\ldots,d$;}\\
s_{d+1}&=&\gamma(s_1,s_2,\ldots,s_{d+1}).
\end{eqnarray*}
Let $G_{Q}$ be the self-similar group defined by this recursion.
Notice that the generating set $S_Q$ is symmetric, because $s_i^{-1}=s_{\gamma(i)}$ for $i=1,2\ldots,d$ and $s_{d+1}^2=e$.

Every $s_i$ for $i=1,\ldots,d$ changes only the first letter in any word over $X$. In this case one usually identifies $s_i$ and $\pi_i$; then we can write $G_Q=\langle \pi_1,\ldots,\pi_d,s\rangle$, where $s=s_{d+1}$ is given by the recursion $s=\gamma(\pi_1,\ldots,\pi_d,s)$.

\begin{theorem}
The action graphs $\Gamma_n$ of the group $G_Q=\langle S_Q\rangle$ satisfy
$\Gamma_{n+1}=\Gamma_n\rp \Gamma$, $n\geq 1$, where $\Gamma$ is the graph of the action
of $Q$ on $X$. In particular, if $\Gamma$ is connected, then all $\Gamma_n$ are connected.
\end{theorem}
\begin{proof}
The graph $\Gamma$ is a $d$-regular graph on $d+1$ vertices, while $\Gamma_n$ are $(d+1)$-regular graphs. In order to define the replacement product $\Gamma_n\rp\Gamma$, we label the edges of $\Gamma_n$ by the vertices of $\Gamma$ as follows.
For $x\in X$, define the
$x$-neighbor of a vertex $v\in X^n$ as $v[x]:=s_x(v)$. In this way we get the labeling
of edges $v
\begin{tikzpicture} \draw (0.1,0.1) -- (1.1,0.1); \node [above] at (0.6,0)
{\scriptsize{$x$ \ $\gamma(x)$}};\end{tikzpicture} s_x(v)$. Now we can consider the
replacement product $\Gamma_n\rp\Gamma$. The vertex set of $\Gamma_n\rp\Gamma$ can be
naturally identified with the vertex set $X^{n+1}$ of $\Gamma_{n+1}$ via $(x,v)\leftrightarrow xv$.
For every edge $x\overset{\pi_i}{\mbox{ --- }} y=\pi_i(x)$ in $\Gamma$, the edges $xv\mbox{ --- } yv$ in $\Gamma_n\rp\Gamma$ coincide with edges in $\Gamma_{n+1}$ given by $s_i$. For every edge $v
\begin{tikzpicture} \draw (0.1,0.1) -- (1.1,0.1); \node [above] at (0.6,0)
{\scriptsize{$x$ \ $\gamma(x)$}};\end{tikzpicture} s_x(v)$ in $\Gamma_n$, the edge $xv\mbox{ --- } \gamma(x)s_x(v)$ in $\Gamma_n\rp\Gamma$ coincide with the edge in $\Gamma_{n+1}$ given by $s_{d+1}$.
\end{proof}

The wreath recursion for the group $G_Q$ defines a finite automaton over $X$ with $d+2$
states. All these automata belong to the important class of bounded automata. In
particular, the groups $G_Q$ belong to the class of contracting self-similar groups
(see \cite{self_sim_groups}). The action graphs $\Gamma_n$ of groups generated by
bounded automata were studied in \cite{PhD_Bond}. In particular, the diameters of
graphs $\Gamma_n$ have exponential growth in terms of $n$, and there is an algorithmic
method to find the exponent of growth as the Perron–Frobenius eigenvalue of certain
non-negative integer matrix.

Some of the groups $G_Q$ were studied before as interesting examples of automaton groups. To see this,
assume that all permutations in $Q$ are involutions (then $\gamma$ is trivial), $d\geq 2$, and the graph $\Gamma$ is connected.
Then the group $G_Q$ is a $\GGS$ group studied in \cite{PhD_Barth,branch_groups}. In particular, in this case $G_Q$ is not finitely presented and has intermediate growth. Is it true that all groups $G_Q$ have subexponential growth?

\begin{example}
Let $d=2$ and $Q=\{ \sigma, \sigma^{-1}\}$, $\sigma=(1\,2\,3)$. The group $G_Q$ is generated by $\sigma$ and $s=(1\,2)(\sigma,\sigma^{-1},s)$. The generating automaton is shown on the right-hand side of Figure~\ref{fig_automata}.
\end{example}

\section{Wreath recursions leading to expanding graphs}

In this section we construct wreath recursions that model the iterations
(\ref{eqn_expanders_zig_zag}) and (\ref{eqn_expanders_rp}).

Fix $k\geq 1$. Let $X=\{1,2,\ldots,d\}$ and $P$ be a symmetric set of permutations on $X$ such that $d=|P|^{2k}$. We introduce formal symbols $s_{(\pi,\tau)}$ for $\pi,\tau\in P$ and define wreath recursion for the set $S_P=\{s_{(\pi,\tau)} : \pi,\tau\in P\}$ as follows. Take all
words of length $k$ over $S_P$, there are $|S_P|^k=d$ such words, and fix an order on
them: $w_1,w_2,\ldots,w_{d}$. Let $\gamma$ be the involution on $X$ such that if
$w_x=s_{(\pi_1,\tau_1)}s_{(\pi_2,\tau_2)}\ldots s_{(\pi_k,\tau_k)}$ then
$w_{\gamma(x)}=s_{(\tau^{-1}_k,\pi^{-1}_k)}\ldots
s_{(\tau^{-1}_2,\pi^{-1}_2)}s_{(\tau^{-1}_1,\pi^{-1}_1)}$. Define wreath recursion by
\[
s_{(\pi,\tau)}=\tau\gamma (w_1,w_2,\ldots,w_d)\pi=\tau\gamma\pi
(w_{\pi(1)},w_{\pi(2)},\ldots,w_{\pi(d)}), \ \pi,\tau\in P.
\]
Let $G_{P,k}$ be the self-similar group generated by this recursion. Notice that
$S_P$ defines a symmetric generating set of $G_{P,k}$, where $s_{(\pi,\tau)}^{-1}=s_{(\tau^{-1},\pi^{-1})}$. When $k=1$ we get the groups $G_P=G_{P,1}$ from the previous section.

Let us consider the associated
action graphs $\Gamma_n$. Note that each word $w_x$ represents a path of length $k$ in
$\Gamma_n$, which is an edge in the graph $\Gamma_n^k$. For $x\in X$, define the
$x$-neighbor of a vertex $v\in X^n$ in $\Gamma_n^k$ as $v[x]:=w_x(v)$. Then each edge
$s_{(\pi,\tau)}(xv)=\tau(\gamma(\pi(x)))w_{\pi(x)}(v)$ in $\Gamma_{n+1}$ is precisely
the zig-zag path in $\Gamma_n^k\zz \Gamma$. We get the following statement.

\begin{theorem}
The action graphs $\Gamma_n$ of the group $G_{P,k}=\langle S_P\rangle$ satisfy
$\Gamma_{n+1}=\Gamma_n^k\zz \Gamma$, $n\geq 1$, where $\Gamma$ is the graph of the
action of $P$ on $X$.
\end{theorem}


If $\lambda(\Gamma)$ and $\lambda(\Gamma_1)$ are small enough (for example, less than
$1/5$), then we get a sequence of expanders. Therefore this construction gives simple
explicit examples of self-similar groups whose graphs $\Gamma_n$ form an expanding
family.

As above the construction can be modified for the case $d>|P|^k$ by adding empty words
to the wreath recursion.

\vspace{0.2cm} Similarly we model the iteration (\ref{eqn_expanders_rp}). Fix $k\geq 1$. Let $Q=\{\pi_1,\ldots,\pi_d\}$ be a symmetric set of permutations on $X=\{1,2,\ldots,(d+1)^k\}$. Let $s$ be a formal symbol and $S_Q=\{\pi_1,\ldots,\pi_d,s\}$, where $\pi_i$ is considered as transformation of $X^{*}$ that changes the first letter of words. Take all
words of length $k$ over $S_Q$, there are $(d+1)^k$ such words, and fix an order on
them: $w_1,w_2,\ldots,w_{(d+1)^k}$. Let $\gamma$ be the involution on $X$ such that if
$w_x=s_{1}s_{2}\ldots s_{k}$ then
$w_{\gamma(x)}=s_{k}^{-1}\ldots s_{2}^{-1}s_{1}^{-1}$, where for the symbol $s_i=s$ we put $s^{-1}:=s$.
We define the self-similar group $G_{Q,k}=\langle \pi_1,\ldots,\pi_d,s\rangle$, where $s$ is given by the wreath recursion \[
s=\gamma(w_1,w_2,\ldots,w_{(d+1)^k}).
\]
The set $S_Q$ defines a symmetric generating set of $G_{Q,k}$, because $Q$ is symmetric and $s^2=e$.

As above we get the following statement.

\begin{theorem}
The action graphs $\Gamma_n$ of the group $G_{Q,k}=\langle S_Q\rangle$ satisfy
$\Gamma_{n+1}=\Gamma_n^k\rp \Gamma$, $n\geq 1$, where $\Gamma$ is the graph of the
action of $Q$ on $X$.
\end{theorem}

This construction produces other examples of self-similar groups whose graphs $\Gamma_n$ form an expanding family when  $k\geq 4$ and $\lambda(\Gamma)$ and $\lambda(\Gamma_1)$ are small enough.

\vspace{0.2cm}\textbf{Questions.} It is interesting what are algebraic and geometric properties of the groups $G_{P,k}$
and $G_{Q,k}$. Are these groups finitely presented? have property (T)? What are their
profinite completions?  What are the properties of their action on the boundary of the space $X^{*}$, i.e., infinite sequences $x_1x_2\ldots$ over $X$?

\bibliographystyle{plain}

\end{document}